\documentclass{article}

\usepackage{amssymb, latexsym,pdfsync,amsmath,amsthm, ulem,hyperref,graphicx}
\usepackage{fullpage}
\usepackage{pgf,tikz}
\usepackage{mathrsfs}
\usetikzlibrary{arrows}
\newtheorem{theorem}{Theorem}

\newtheorem{lemma}{Lemma}
\newtheorem{proposition}{Proposition}

\theoremstyle{definition}
\newtheorem{definition}{Definition}
\newtheorem{remark}{Remark}

\newcommand{\SSS}{\mathbb{S}}

\newcommand{\FF}{\mathbb{F}}

\newcommand{\Fq}{\mathbb{F}_q}
\newcommand{\Fqn}{\mathbb{F}_{q^n}}

\newcommand{\D}{\mathcal D}

\newcommand{\cI}{\mathcal I}
\newcommand{\cP}{\mathcal P}
\newcommand{\cL}{\mathcal L}

\newcommand{\cH}{\mathcal H}

\newcommand{\dinv}{d^{\mathrm{inv}}}

\def\Fq{{\mathbb{F}}_q}

\def\T{\mathbb{T}}

\def\PG{\mathrm{PG}}
\def\AG{\mathrm{AG}}

\def\dim{\mathrm{dim}}

\newtheorem{conjecture}{Conjecture}

\newcommand{\rank}{\mathrm{rank}}

\begin{document}
\title{On Translation Hyperovals in Semifield Planes}

\author{Kevin Allen and John Sheekey}

\maketitle

\begin{abstract}
In this paper we demonstrate the first example of a finite translation plane which does not contain a translation hyperoval, disproving a conjecture of Cherowitzo. The counterexample is a semifield plane, specifically a Generalised Twisted Field plane, of order $64$. We also relate this non-existence to the covering radius of two associated rank-metric codes, and the non-existence of scattered subspaces of maximum dimension with respect to the associated spread.
\end{abstract}

\section{Introduction}

Hyperovals are extremal combinatorial objects in projecive planes; namely a hyperoval is a set of $q+2$ points in which no three lie on a common line. We refer to Section \ref{sec:def} for formal definitions. 

Hyperovals have attracted much attention over the years, particularly in the case of Desarguesian planes. Papers regarding existence, construction, and classification abound. While full classifications appear out of reach, the addition of extra assumptions on the symmetry of the plane, the hyperoval, or both,  leads to interesting questions and more potential for classification.

In particular, the most well-studied  non-Desarguesian planes are the {\it translation planes}, while the best understood hyperovals in Desarguesian planes are the {\it translation hyperovals}. Therefore it is natural to consider the question of existence and classification for translation hyperovals in translation planes.

It is known that translation hyperovals exist in certain translation planes; for example Desarguesian planes \cite{Payne1971}, Andr\'e  planes \cite{Denniston}, Hall planes \cite{Korchmaros}, and Knuth's binary semifield planes \cite{DuTrZh}. Cherowitzo \cite{Cherowitzo}  computationally classified all hyperovals (translation and otherwise) in each of the nine translation planes of order $16$. In particular he showed that every translation plane of order $16$ contains a translation hyperoval, which lead him to make the following statement.

\begin{conjecture}\label{con:cher}
{\it These results... lead one to the natural conjecture that translation hyperovals exist in all translation planes of even order.}
\end{conjecture}

In this paper we will disprove this conjecture, by exhibiting a projective plane of order $64$ containing no translation hyperoval. Specifically, we show that the {\it twisted field plane} of order $64$, which is a {\it semifield plane}, contains no translation hyperovals. We will also relate this problem to the (non-)existence of so-called {\it scattered subspaces} with respect to a spread associated to the translation plane, as well as the covering radius of the rank-metric code (spread set) associated to the translation plane.

\section{Definitions and Background}\label{sec:def}

In this section we collect the necessary definitions and background for this article. We refer to \cite{Dembowski,HughesPiper, Cherowitzo, Handbook,LavrauwScat} for further details on these largely well-known topics.

\subsection{Projective Planes}

A {\it projective plane} $\pi$ is an incidence structure $(\cP,\cL,\cI)$ consisting of a set of {\it points} $\cP$, a set of lines $\cL$, and an incidence relation $\cI\subset \cP\times \cL$ such that every two distinct points are both incident with precisely one common line, and every pair of distinct lines are incident with precisely one common point. The {\it dual} of $\pi$, denoted $\pi^d$, is the incidence structure $(\cL,\cP,\cI^d)$, where $\cI^d$ is the reverse relation of $\cI$. Since the terms points and lines are interchangeable in the definition of a projective plane, $\pi^d$ is again a projective plane.

It is well known that for any finite projective plane there exists a positive integer $q$, called the {\it order} of $\pi$, such that the plane contains $q^2+q+1$ points and $q^2+q+1$ lines, every line is incident with $q+1$ points, and every point is incident with $q+1$ lines. If $q$ is a prime power, there exists a projective plane of order $q$; the converse is a famous open problem. 

The classical examples of projective planes are the {\it Desarguesian planes} $\PG(2,F)$ where $F$ is some (skew) field, in which points are one-dimensional vector subspaces of $V(3,F)$, lines are two-dimensional vector subspaces of $V(3,F)$, and incidence is given naturally by inclusion. This plane can also be realised as the {\it completion} of the {\it affine plane} $\AG(2,F)$, in which points are vectors in $V(2,F)$, and lines are translations of one-dimensional subspaces; that is, cosets $u+\langle v\rangle$ for some $u,v\in V(2,F),v\ne 0$. Then the addition of a {\it line at infinity} consisting of the {\it directions} $\langle v\rangle$ returns a projective plane isomorphic to the previous description.

\subsection{Translation Planes}\label{ssec:tplane}

Translation planes are projective planes with extra symmetry. They arise from affine planes sharing some natural properties of $\AG(2,F)$. In order to introduce them formally, we need to recall some technical terminology.

An {\it isomorphism} from a plane $\pi_1 = (\cP_1,\cL_1,\cI_1)$ to another $\pi_2 = (\cP_2,\cL_2,\cI_2)$ is a bijection $\phi$ from $\cP_1\cup\cL_1$ to $\cP_2\cup\cL_2$ which preserves type and incidence; that is,
$\phi(\cP_1)=\cP_2$, $\phi(\cL_1)=\cL_2$,
and $p\in \ell\Leftrightarrow \phi(p)\in \phi(\ell)$.

A {\it collineation} of a plane $\pi$ is an isomorphism from $\pi$ to itself; a {\it correlation} is an isomorphism from $\pi$ to its dual $\pi^d$. An {\it elation} of $\pi$ with centre $p$ and axis $\ell\ni p$  is a collineation of $\pi$ fixing every point on $\ell$ and every line containing $p$.

A projective plane $\pi$ is said to be a {\it translation plane} if there exists a line $\ell$ such that the group of elations with axis $\ell$ acts transitively on the points of $\pi\backslash \ell$. If the plane is not Desarguesian, then the line $\ell$ is unique, and is called the {\it translation line} of $\pi$. We will usually denote the translation line by $\ell_\infty$. Any collineation of $\pi$ must then fix $\ell_\infty$.

The name {\it translation plane} can be more easily understood in the affine setting; in the affine plane $\AG(2,F)$, a translation $\tau_u:v\mapsto v+u$ clearly maps lines to lines,  preserves direction, and fixes all lines with direction $\langle u\rangle$. The natural extension of this map to $\PG(2,F)$ then satisfies the definition of an elation given above, with $\ell$ the line at infinity, and $p$ the point at infinity corresponding to $\langle u\rangle$. Furthermore, the group of translations clearly acts transitively on the points of $\AG(2,F)$. Thus the concept of an elation and a translation plane is a natural generalisation of this example.

The dual of a translation plane is not necessarily a translation plane. If both a plane $\pi$ and its dual $\pi^d$ are translation planes, then we call it a {\it semifield plane}. If a semifield plane $\pi$ is not Desarguesian, then there is a unique point $p_\infty\in \pi$ such that the dual of $p$ is the translation line of $\pi^d$. We call this point the {\it shears point}. Any collineation of $\pi$ must then fix $p_\infty$. It is known that we must have $p_\infty\in \ell_\infty$, that is, the shears point must lie on the translation line. Furthermore, the collineation group of a semifield plane has precisely three orbits on points: the shears point, the points of the translation line other than the shears point, and the remaining points of the plane.

\subsection{Hyperovals}

A {\it hyperoval} in a finite projective plane $\pi$ (of even order $q$) is a set $\cH$ of $q+2$ points such that no three points of $\cH$ are incident with a common line. Hyperovals can only exist in planes of even order; in planes of odd order, the maximum size of a set with this property is $q+1$, and a famous result of Segre tells us that all sets attaining this bound are equivelent to the set of points of a {\it conic}. 

The study of hyperovals in planes $\PG(2,\Fq)$, $q$ even, has a long history, with connections to important objects in coding theory, namely {\it MDS codes} with certain parameters and properties. We refer to \cite{Vandendriessche} for an up-to-date list of the known constructions for Desarguesian planes, as well as the largest plane with a complete computer classification. Hyperovals in general planes have also received much attention. It was conjectured that every projective plane of even order contained a hyperoval; this was disproved by the computer classification of \cite{Royle}, where a projective plane of order $16$ containing no hyperovals was exhibited.

We are now ready to formally introduce translation hyperovals.

\begin{definition}
    A {\it translation hyperoval} is a hyperoval $\cH$ such that there exists a line $\ell$ for which the group of elations with axis $\ell$ acts transitively on $\cH\backslash \ell$.
\end{definition}

Payne \cite{Payne1971} showed that every translation hyperoval in $\PG(2,2^n)$ is equivalent to one defined by the set of vectors $\{(0,1,0),(0,0,1),(1,x,x^{2^i}):x\in \FF_{2^n}\}$ for some positive integer $i$ relatively prime to $n$. These translation hyperovals were first constructed by Segre \cite{Segre1957}.

 It is known (see e.g. \cite{Cherowitzo}) that for a translation hyperoval in a (non-Desarguesian) translation plane, the line $\ell$ must be the translation line $\ell_\infty$, and $|\cH\cap \ell_\infty|=2$.

Since a semifield posesses a distinguished point at infinity, namely the shears point, it makes sense to consider whether or not a translation hyperoval contains the shears point.

\begin{definition}
    A translation hyperoval in a semifield plane is said to be of {\it shears type} if it contains the shears point, and of {\it non-shears type} if it does not.
\end{definition}

Translation hyperovals in translation planes were studied by Cherowitzo in \cite{Cherowitzo} where he computationally classified all hyperovals (translation and otherwise) in each of the nine translation planes of order $16$. In particular he showed that every translation plane of order $16$ contains a translation hyperoval, which lead him to Conjecture \ref{con:cher}.


\section{Quasifields, Spreads, Spread Sets and Translation Planes}

In this section we outline various well-known correspondences between quasifields, spreads, spread sets and Translation Planes. We refer to \cite{Handbook} for details, proofs, and further references. 

\subsection{Quasifields and Semifields}

A {\it quasifield} is an algebraic structure similar to a finite field, without the requirement that multiplication be associative, and assuming only one distributive law. In the finite case, a quasifield must have order $q^n$ for some prime power $q$ and some positive integer $n$. In this case we may take the additive structure to be $(\Fqn,+)$, with multiplication $\circ:\Fqn\times \Fqn\rightarrow \Fqn$ satisfiying
\begin{itemize}
\item    $(x+x')\circ y=x\circ y+x'\circ y$ for all $x,x',y\in \Fqn$;
\item
For every $a,b\in \Fqn$, $a\ne 0$, there exist unique $x,y\in \Fqn$ such that $x\circ a=a\circ y=b$.
\end{itemize}

A {\it semifield} is a quasifield in which the second distributive law also holds:
\begin{itemize}
\item    $x\circ (y+y')=x\circ y+x\circ y'$ for all $x,y,y'\in \Fqn$.
\end{itemize}

Note that in the literature quasifields and semifields are assumed to contain a multiplicative identity, with the terms {\it prequasifield} and {\it presemifield} used to describe the case where an idenitity is not assumed. Since this distinction does not have any relevance for this paper, we will abuse terminology and drop the prefix. 

There are many known constructions for quasifields and semifields. We refer to \cite{Handbook,LaPo} for examples of constructions, and \cite{Rua2009} for the classification of semifields of order $64$.

One of the most well-studied families, and the example most relevant to this paper, are the {\it generalised twisted fields} of Albert \cite{Albert1961}. These are semifields with multiplication
\[
x\circ y := xy-jx^{q^i}y^{q^k}, 
\]
where $j$ is a fixed element of $\Fqn$ satisfying $N_{\Fqn:\FF_{q^{(n,i,k)}}}(j)\ne 1$, with $(n,i,k)$ denoting the greatest common divisor of these three integers. For example, for $q=2,n=6,i=2,k=4$, the multiplication
\[
x\circ y := xy-jx^{2^2}y^{2^4}, 
\]
defines a semifield if and only if $N_{\FF_{2^6},\FF_{2^2}}(j)\ne 1$. Such elements certainly exist, for example any $j$ satisfying $j^6+j+1=0$.

Two semifields with respective multiplications $\circ$ and $\star$ are {\it isotopic} if there exist invertible additive maps $A,B,C$ from $\Fqn$ to itself such that $A(x\circ y)=B(x)\star C(y)$ for all $x,y\in \Fqn$. Every presemifield is isotopic to a semifield via {\it Kaplansky's trick} \cite{LaPo}.



\subsection{Spreads and Spread Sets}

A {\it spread} (or {\it $n$-spread}) in $V=V(2n,F)$ is a set $\D$ of $n$-dimensional vector subspaces of $V$ such that every nonzero element of $V$ is contained in precisely one element of $\D$. 

Let us identify the elements of $V(2n,q)$ with elements of $V(2,q^n)$, and let $S_\infty= \{(0,x):x\in \Fqn\}$. The $n$-dimensional spaces meeting $S_\infty$ trivially are precisely those of the form 
\[
S_f := \{(x,f(x)):x\in \Fqn\},
\]
where $f(x)\in \Fqn[x]$ is a {\it linearised polynomial}, i.e. a polynomial of the form $f(x)=\sum_{i=0}^{n-1}f_ix^{q^i}$. These are the polynomials which define $\Fq$-linear maps from $\Fqn$ into itself. We denote the set of linearised polynomials of degree at most $q^{n-1}$ as $\cL$.

To a spread $\D$ containing $S_\infty$ we can associate a unique set of linear maps (or linearised polynomials) $C(\D)$ by
\[
C(\D) = \{f:S_f\in \D\}.
\]
These satisfy the property that $C(\D)=q^n$ and for any $f,g\in C(\D),f\ne g$
, we have $\rank(f-g)=n$, where $\rank$ denotes the usual linear algebra rank of an $\Fq$-linear map. This is called the {\it spread set} associated to $\D$. The definition of a spread set coincides with that of a (not necessarily linear) {\it maximum rank distance (MRD) code}.
Conversely, any set $C$ satisfying this property defines a spread by
\[
\D(C) := \{S_\infty\}\cup \{S_f:f\in C\}.
\]

For a linear map $f$, we can define 
\[
\phi_f:(x,y)\mapsto (x,f(x)+y).
\]
If $C$ is additively closed, then each of the maps in the set $\phi_C := \{\phi_f:f\in C\}$ fix the spread $\D(C)$. Moreover, each $\phi_f$ fixes $S_\infty$ pointwise, and $\phi_C$ is an abelian group acting transitively on $\D(C)\backslash \{S_\infty\}$. The spread is then called a {\it semifield spread}, and $C$ a {\it semifield spread set}, for reasons that will shortly become apparent. We refer to the distinguished element $S_\infty$ as the {\it shears element} of $\D$. 

Given a quasifield $Q$, we can define a spread set and a spread as follows. Define
\[
R_y(x):= x\circ y.
\]
Then each $R_y$ is an additive map. Moreover $R_y-R_{y'}$ is invertible for all $y\ne y'$, since otherwise there would exist some nonzero $a$ such that $a\circ y=a\circ y'$, contradicting one of the axioms of a quasifield. Thus
\[
C(Q) := \{R_y:y\in \Fqn\}
\]
defines a spread set, and $\D(Q) := \D(C(Q))$ is a spread. If $Q$ is a semifield, then $R_{y+y'}=R_y+R_{y'}$, and so $C(Q)$ is additively closed.

Conversely, from a spread or spread set we can define a quasifield; note however that the quasifield is not uniquely determined by the spread or spread set. However it is uniquely determined up to isotopy.


\subsection{Translation Planes from Spreads}
\label{ssec:tplanespread}

From a spread $\D$ we can define an affine plane as follows:
\begin{itemize}
    \item Points: elements of $V(2n,q)$;
    \item Lines: translations of elements of $\D$, i.e. $u+S$ for $u\in V(2n,q),S\in \D$.
\end{itemize}

We can complete this to a projective plane by adding a line at infinity $\ell_\infty$, whose points are the elements of $\D$; to a line $u+S$ we add the point at infinity $S$. We denote this plane as $\pi(\D)$.

It is straightforward to check that this is indeed a translation plane. Moreover, Andr\'e showed that every translation plane arises from a spread \cite{Andre}. 

By the discussion in the previous subsections, we can define a translation plane from a quasifield, and from a spread set. We may denote these naturally as $\pi(Q)$, $\pi(C)$ respectively. 

The dual of a translation plane defined by a semifield is again a translation plane, and so the plane is a semifield plane. The shears point corresponds to the shears element of the spread. As mentioned in Section \ref{ssec:tplane}, the collineation group of a semifield plane has precisely three orbits on points: the shears point, the points of the translation line other than the shears point, and the remaining points of the plane. The transitivity on the points of the translation line other than the shears point is demonstrated by the maps $\phi_f$ defined in the previous section.

It is well known that two semifields define isomorphic planes if and only if they are isotopic.

\subsection{Scattered Subspaces}

It is known (see e.g. \cite{LavrauwScat,DHVdV2020}) that translation hyperovals correspond to {\it scattered subspaces} with respect to spreads. We recall the relevant notions and demonstrate this fact here.

A subspace $U$ is said to be {\it scattered} with respect to a spread $\D$ if 
\[
\dim(U\cap S)\leq 1\quad\mathrm{for~all }~S\in \D.
\]

The study of scattered subspaces with respect to a spread originates from \cite{BlLa2000}; we refer to \cite{LavrauwScat} for a recent survey of the various applications of scattered subspaces, including that of translation hyperovals. We note that the interest in this notion is not restricted to spreads in $V(2n,q)$; the more general case of spreads of $n$-dimensional subspaces of $V(kn,q)$ is also of interest. However, here we deal only with the case $k=2$. In this case it is straightforward to obtain an upper bound on the dimension of a scattered subspace.

\begin{proposition}\cite[Lemma 3.1 and Theorem 4.1]{BlLa2000}
Let $\D$ be a spread in $V(2n,q)$. Then
the dimension of a scattered subspace is at most $n$. Moreover, if $\D$ is Desarguesian, then there exists a scattered subspace of dimension $n$.
\end{proposition}

Suppose $U$ is a scattered subspace of dimension $n$ with respect to a spread in $V(2n,2)$. Then $U$ must intersect $2^n-1$ distinct elements of $\D$ nontrivially, and hence since $|\D|=q^n+1$, there exist precisely two elements of $\D$ which intersect $U$ trivially, say $S_1$ and $S_2$. We claim that the set of points of the translation plane $\pi(\D)$ defined by the affine points $U$ and the two points at infinity $S_1,S_2$ form a translation hyperoval, which we will denote by $\cH_U$. Consider a line $v+S$ of $\pi(\D)$ with $S\notin\{S_1,S_2\}$. Then the points of intersection of $v+S$ and $\cH_U$ are affine, and so the number of them is equal to $|(v+S)\cap U|$, which is either $0$ or $2$, since $|U\cap S|=2$. Next consider a line $v+S_i$, $i\in \{1,2\}$. Then this line meets $\cH_U$ in the point at infinity corresponding to $S_i$, as well as unique affine point $v+s_i$, where $s_i+u_i=v$ for a unique $u_i\in U$. The uniqueness here follows from the fact that $V(2n,2) = S_i\oplus U$. Finally the line at infinity clearly meets $\cH_U$ in the two points $\{S_1,S_2\}$, proving that $\cH$ is a hyperoval. Then the group of translations $\tau_u:u\in U$ clearly acts transitively on $U=\cH_U\backslash\ell_\infty$, showing that $\cH_U$ is indeed a translation hyperoval.

Conversely, it has been shown that (up to equivalence) every translation hyperoval arises in this way from a scattered subspace. 

\begin{proposition}\cite[Theorem 1.7]{LavrauwScat}
Suppose $\cH$ is a translation hyperoval in a translation plane $\pi$ of order $2^n$, and suppose $\D$ is a spread in $V(2n,2)$ such that $\pi$ is isomorphic to $\pi(\D)$. Then there exists a translation hyperoval in $\pi$ if and only if there exists an $n$-dimensional scattered subspace with respect to $\D$.
\end{proposition}

Hence the existence of translation hyperovals corresponds to the existence of scattered subspaces. Recently in \cite{GrRaShZu} the following lower bound on the largest dimension of a scattered subspace was shown.

\begin{proposition}\cite[Proposition 4.15]{GrRaShZu} 
Let $\D$ be a spread in $V(2n,q)$, $q\geq 8$. Then
there exists a scattered subspace of dimension at least $n/2 - 1$.
\end{proposition}

Clearly this does not guarantee the existence of translation hyperovals, though it does indicate the dimension at which scattered subspaces become hard to find, and impossible to guarantee by combinatorial methods.

Various generalisations of the notion of scattered subspaces have been put forward in recent years, such as {\it evasive subspaces} \cite{CsEvasive}. One which is of particular relevance to this paper is that of $(\D,h)$-scattered subspaces:

A subspace $U$ is said to be {\it $(\D,h)$-scattered} if 
\[
\dim(U\cap S)\leq h\quad\mathrm{for~all }~S\in \D.
\]

Clearly the property of a subspace being $(\D,1)$-scattered coincides with it being scattered with respect to $\D$. The following specialisation of \cite[Theorem 6.7]{GrRaShZu} guarantees the existence of $(\D,h)$-scattered subspaces in certain circumstances.

\begin{proposition}
Let $\D$ be a spread in $V(2n,q)$. Then
there exists a $(\D,h)$-scattered subspace of dimension $n$ for any $h\geq \lceil \sqrt{n+1}-1\rceil$ if $q\geq 4$, and for any $h\geq \lceil \sqrt{n+2}-1\rceil$ if $q< 4$.
\end{proposition}

For the case $n=6$, which is the case for spreads arising from semifields of order $q^6$, we obtain that there exists a $2$-scattered subspace of dimension $6$ with respect to any $6$-spread $\D$ in $V(12,q)$. In particular, we note that the existence of a $1$-scattered subspace of dimension $n$, and hence a translation hyperoval, is not guaranteed. Indeed, we will demonstrate a counterexample.

\subsection{Covering Radius}

Let us assume that $\D=\D(\SSS)$, where $\SSS$ is a semifield. Suppose $U$ is an $n$-dimensional scattered subspace with respect to $\D$. Then without loss of generality we may assume that one of the following two occur.

\begin{itemize}
    \item[](Shears Type)~~~~~ ~$U\cap S_\infty = U\cap S_0=0$;
    \item[](Non-Shears Type) $U\cap S_\infty \ne 0, U\cap S_0=U\cap S_y=0$ for some $y\ne 0$.
\end{itemize}

These two cases correspond to whether or not the translation hyperorval $\cH_U$ contains the shears point $S_\infty$.

The {\it covering radius} of a set or subspace $C$ of the space of linear maps $M=\mathrm{End}_{\Fq}(\Fqn)$ is a notion arising naturally from coding theory in the rank metric; we refer to  \cite{ByrneRav} for background. We denote it by $\rho(C)$ and define it as
\[
\rho(C)= \min\{i:\forall g\in M,\exists f\in C \mathrm{ ~s.t.~}\rank(f-g) \leq i\}.
\]

For a spread set $C\subset M$, which has cardinality $q^n$ and is such that every nonzero element of $C$ is invertible, we define $C^{-1}= \{f^{-1}:f\in C,f\ne 0\}\cup\{0\}$. If $C$ is a spread set, then $\rho(C)\leq n-1$, since otherwise there would exist $g\notin C$ such that $\rank(g-f)=n$ for all $f\in C$. But for any nonzero $a\in \Fqn$ there exists a unique $y\in \Fqn$ such that $a\circ y = g(a)$, and so there exists $R_y\in C$ such that $\rank(g-R_y)<n$.

\begin{theorem}\label{thm:covering}
Let $\SSS$ be a semifield of order $2^n$, $\pi(\SSS)$ the translation plane it defines, and $C=C(\SSS)$ the spread set it defines in $V(2n,2)$. Then there exists a translation hyperoval of shears type in $\pi(\SSS)$ if and only if $\rho(C)=n-1$, and there exists a translation hyperoval of non-shears type in $\pi(\SSS)$ if and only if $\rho(C^{-1})=n-1$. 
\end{theorem}

\begin{proof}
The plane $\pi(\SSS)$ contains a translation hyperoval of shears type if and only if there exists an $n$-dimensional subspace $U$ which is scattered with respect to $\D(\SSS)$ such that $U\cap S_\infty=0$. For any $n$-dimensional subspace such that $U\cap S_\infty=0$ there exists an $\Fq$-linear map $f$ from $\Fqn$ to itself such that 
\[
U=\{(x,f(x)):x\in \Fqn\}.
\]

Let $y\in \Fqn$. Then $U\cap S_y = \{(x,f(x))|f(x)=R_y(x)\}$, and so $\dim(U\cap S_y) = n-\rank(f-R_y)$. Hence there exists an $n$-dimensional subspace which is scattered with respect to $\D(\SSS)$ if and only if there exists $f$ such that $\rank(f-R_y)\geq n-1$ for all $y\in \Fqn$, if and only if $\rho(C)\geq n-1$, if and only if $\rho(C)=n-1$, proving the first claim.

Similarly, $\pi(\SSS)$ contains a translation hyperoval of shears type if and only if there exists an $n$-dimensional subspace $U$ which is scattered with respect to $\D(\SSS)$ such that $U\cap S_0=0$. For any such subspace there exists an $\Fq$-linear map $f$ from $\Fqn$ to itself such that $U=\{(f(x),x):x\in \Fqn\}$. Then for any nonzero $y\in \Fqn$ we have $\dim(U\cap S_y)= n-\rank(f-R_y^{-1})$, while $\dim(U\cap S_\infty)=n-\rank(f)$, and so arguing as before we obtain the second claim.
\end{proof}

Note that we do not have such a result for general quasifields, since the lack of a distinguished element at infinity nor transitivity on the remaining points at infinity means that there is not a {\it canonical} choice for the spread set.

Note also that the connection between the existence of scattered subspaces with respect to semifield spreads and the covering radius of $C$ and $C^{-1}$ are also valid for $q>2$; however in this case we do not obtain translation hyperovals.

\subsection{Linearised Polynomials and Dickson Matrices}

In order to explicitly determine whether or not translation hyperovals exist, we need a practical method for determining the existence of a linear map $f$ such that $\rank(f-R_y)\geq n-1$ for all $y\in \Fqn$. We do this by utilising linearised polynomials and Dickson matrices; this is the approach used by Payne to classify translation hyperovals in Desarguesian planes, and also used productively in recent years in the construction of MRD codes.

To a linearised polynomial $f(x)= \sum_{i=0}^{n-1}f_ix^{q^i}$ we associate the {\it Dickson (or autocirculant) matrix} $D_f$ defined as follows:
\[ D_f :=    \begin{pmatrix}
            f_{0} & f_{1}  & \cdots  & f_{n-1}\\
            f_{n-1}^q & f_{0}^q  & \cdots & f_{n-2}^q\\
            \vdots & \ddots& \ddots & \vdots\\
            f_{1}^{q^{n-1}}  & f_{2}^{q^{n-1}} & \cdots  & f_{0}^{q^{n-1}}
        \end{pmatrix}
\]

It is well known that the assignment $f\mapsto D_f$ is linear, and $\rank(f) = \rank(D_f)$ \cite{MenichettiAffine}. Hence we can translate Theorem \ref{thm:covering} to this setting.

\begin{lemma}\label{lem:dickson}
The plane $\pi(\SSS)$ contains a translation hyperoval of shears type if and only if there exists $f\in \cL\backslash C(\SSS)$ such that
\[
\rank(D_{R_y}-D_f)\geq n-1
\]
for all $y\in \Fqn$.

The plane $\pi(\SSS)$ contains a translation hyperoval of non-shears type if and only if there exists $g\in \cL\backslash C(\SSS)$ such that
\[
\rank(D_{R_y}^{-1}-D_g)\geq n-1
\]
for all $y\in \Fqn^\times$ and $\rank(D_g)=n-1$.
\end{lemma}

Now the entries of $D_{R_y}$ are polynomials in $y$, and since $\det(D_{R_y})=1$ for all non-zero $y$, we can regard $D_{R_y}^{-1}$ as a matrix whose entries are polynomials in $y$. Hence for $f,g\in \cL$, 
the functions $d_f(y) := \det(D_{R_y}-D_f)$ and $\dinv_g(y) := \det(D_{R_y}^{-1}-D_g)$ are both polynomials in $y$ (whose coefficients are expressions in the unknown coefficients of $f$ and $g$ respectively). Note that $d_g(0)=\det(D_g)$. If necessary we can replace $d_f(y)$ and $\dinv_g(y)$ with their reduction modulo $y^{2^n}-y$. 

\begin{lemma}\label{lem:dicksondet}
The plane $\pi(\SSS)$ of order $2^n$ contains a translation hyperoval of shears type if and only if there exists $f\in \cL\backslash C(\SSS)$ such that $d_f(y) = y^{2^n-1}+1$, and contains a translation hyperoval of non-shears type if and only if there exists $g\in \cL\backslash C(\SSS)$ such that $\rank(g)=n-1$ and $\dinv_g(y) = \frac{y^{2^n}+y}{y+a}$ for some $a\in \Fqn^\times$.
\end{lemma}

\begin{proof}
Let $U=\{(x,f(x)):x\in \Fqn\}$, and suppose $U$ defines a translation hyperoval of shears type. Then we may assume without loss of generality that $U\cap S_0=0$, and so $U\cap S_y\ne 0$ for all $y\ne 0$. Then $\rank(D_{R_y}-f)=n-1$ for all nonzero $y\in \Fqn$, and so $d_f$ is zero at all nonzero elements of $\Fqn$ and nonzero at $y=0$. Clearly this implies that $d_f(y)=y^{2^n-1}+1$.

Now let $W=\{(g(x),x):x\in \Fqn\}$, and suppose $W$ defines a translation hyperoval of non-shears type. Then we may assume without loss of generality that $W\cap S_0=0$, and there exists a unique $a\in \Fqn^\times$ such that $W\cap S_a=0$ and $W\cap S_y\ne 0$ for all nonzero $y\ne a$. Furthermore $W\cap S_{\infty}\ne 0$, and so $\dinv_g(0)=0$. Hence $\dinv_g(y)\ne 0$ if and only if $y=a$, and so $\dinv_g(y) = \frac{y^{2^n}+y}{y+a}$ as claimed.
\end{proof}

Note that this is the approach used by Payne in his classification of translation hyperovals in $\PG(2,q^n)$. He showed that for the case $R_y(x)= xy$, the requirement that $d_f(y)=y^{2^n-1}+1$ implies that $f(x)$ is monomial, that is, $f(x)=f_ix^{2^i}$. The classification of translation hyperovals in $\PG(2,q^n)$ then follows easily.





\section{Translation Hyperovals in the Generalised Twisted Field plane of order $64$}

In this section we analyse the conditions from Lemma \ref{lem:dickson} and Lemma \ref{lem:dicksondet} for the case of the Generalised Twisted Field plane of order $64$. We choose this plane due to the fact that the Dickson matrices $D_{R_y}$ and $D_{R_y^{-1}}$ are sparse, making the equations manageable. Furthermore the known symmetries of these semifields allow us to further reduce the necessary computation.

The multiplication in this presemifield, which we will denote by $\T$, is given by
\[
x\circ y = xy-jx^{2^2}y^{2^4}=:R_y(x),
\]
where $j$ is a solution to $j^6+j+1=0$. We choose this representation to match that in \cite{Rua2009}. Any semifield isotopic to this presemifield has centre isomorphic to $\FF_4$. In particular, each of the maps $R_y$ are $\FF_4$-linear.

Note furthermore that we have the following identity, which will prove useful in the subsequent calculations:
\[
\alpha R_y (\beta x) = R_{\alpha\beta y}(x)
\]
for all $\alpha,\beta\in \FF_{2^6}$ such that $\alpha\beta^{2^2}=\alpha^{2^4}\beta\ne 0$. Hence for any $\alpha,\beta$ satisfying this condition, the map $\phi_{\alpha,\beta}:(x,y)\mapsto (\beta^{-1}x,\alpha y)$ fixes $\D(\T)$; in particular, it fixes $S_\infty$ and $S_0$, and maps $S_y$ to $S_{\alpha\beta y}$. Note that any such $\phi_{\alpha,\beta}$ fixes one further element of $\D(\T)$ if and only if it fixes every element of $\D(\T)$, and this occurs precisely if $\alpha^9=1$ and $\alpha\beta=1$.

Furthermore, letting $U_f=\{(x,f(x)):x\in \Fqn\}$ and $W_g=\{(g(x),x):x\in \Fqn\}$, we get that $\phi_{\alpha,\beta}(U_f) = U_h$ where $h(x) = \alpha f(\beta x)$, and $\phi_{\alpha,\beta}(W_g) = W_k$ where $k(x) = \beta^{-1} g(\alpha^{-1} x)$.

\subsection{Shears Type}

Suppose $\pi(\mathbb{T})$ contains a translation hyperoval of shears type. By Lemma \ref{lem:dicksondet}, we require the existence of some $f\in \cL\backslash C(\SSS)$ such that $d_f(y) := \det(D_{R_y}-D_f)=y^{2^6-1}+1$. This leads to a system of equations in six unknowns $f_0,\ldots,f_5$ over $\FF_{2^6}$. Note furthermore that for any $\alpha,\beta\in \FF_{2^6}$ such that $\alpha\beta^{2^2}=\alpha^{2^4}\beta\ne 0$, if $h(x)=\alpha f(\beta x)$, then
\[
d_f(y)=d_h(\alpha\beta y)= d_h(y),
\]
and so $f$ defines a translation hyperoval if and only if $h$ defines a translation hyperoval. Note that $h_0=\alpha\beta f_0$, and the set $\{\alpha\beta:\alpha,\beta\in \FF_{2^6},\alpha\beta^{2^2}=\alpha^{2^4}\beta\ne 0\}$ is precisely the set of solutions to $x^{21}=1$. Thus we may assume without loss of generality that $f_0\in \{0,1,j,j^2\}$. Furthermore if $\alpha\beta=1$ then $\beta^9=1$ and $h_1=\beta f_1$, and so we can assume without loss of generality that $f_1^8=f_1$, i.e. $f_1\in \FF_8$.

From the coefficients of $y^{62}$ and $y^{58}$ respectively, we get that 
 \begin{align*}   
    0&=j^{21}f_{0}+j^{38}f_{2}^{4},\\   
   0&=j^{22}f_{0}^{16}f_{4}^{4} + j^{21}f_{0}^5 + j^{22}f_{2}^{20} + j^{21}f_{2}f_{4}^4.
\end{align*}
Thus we have that either $f_0=f_2=0$, or $f_{0}=j^{17}f_{2}^{4}$ and $f_{4}=j^{16}f_{2}^{52}$.


We plug these expressions into the coefficients of $y^{57}$ and $y^{54}$ and set them equal to zero. It turns out that we get the same pair of equations regardless of whether or not $f_2=0$, and we also observe that $f_2$ does not appear in either of the resulting equations: 

\begin{align*}
0&=j^{39}f_{1}^{16}f_{3}^{8} + f_{1}^{2}f_{5}^{4} + j^{5}f_{3}^{16}f_{5}^{2} + j^{34}f_{3}^{4}f_{5}^{8},\\
    0&=j^{10}f_{1}^{33} + j^{17}f_{1}^{12} + f_{3}^{9} + j^{27}f_{5}^{36}.
\end{align*}

Taking into account that $f_1\in \FF_8$, and raising to an appropriate power of $2$, we get that
\begin{align}\label{eqn:shearsystem1}
      0&=f_1(j^{51}f_{3}^{4} + f_{5}^{2}) + j^{34}f_{3}^{8}f_{5} + j^{17}f_{3}^{2}f_{5}^{4},\\
    0&=j^{36}f_{1}^{5}  + f_{3}^{9} + j^{27}f_{5}^{36}.\nonumber
\end{align}

This leads to the following.

\begin{theorem}\label{thm:shearsnotexist}
    The Generalised Twisted Field plane of order $64$ does not contain a translation hyperoval of shears type.
\end{theorem}

\begin{proof}
The following MAGMA code verifies that the system (\ref{eqn:shearsystem1}) has no nontrivial solutions.
{\small
\begin{quote}
\begin{verbatim}
q := 2;
F := GF(q);
P<x> := PolynomialRing(F);
L<j> := ext<F|x^6+x+1>;
F8 := {x:x in L|x^8 eq x};

S<f1,f3,f5> := PolynomialRing(L,3);

g := f1*(j^51*f3^4 + f5^2) + j^34*f3^8*f5 + j^17*f3^2*f5^4;
h := j^36*f1^5  + f3^9 + j^27*f5^36;

s1 := {[a,b,c]:a in F8,b,c in L|Evaluate(g,[a,b,c]) eq 0};
s2 := {[a,b,c]:a in F8,b,c in L|Evaluate(h,[a,b,c]) eq 0};

s1 meet s2 eq {[L|0,0,0]};
\end{verbatim}
\end{quote}}

Hence we have that $f_1=f_3=f_5=0$, implying $f(x)$ is in fact an $\FF_4$-linear map. But since each $R_y$ is also $\FF_4$-linear, then the rank of $f-R_y$ as an $\FF_2$-linear map must be even; in particular it cannot be $n-1=5$, contradicting Theorem \ref{thm:covering}.

Hence this plane does not contain a translation hyperoval of shears type.
\end{proof}

The MAGMA code used in this proof runs in less than one second.

\subsection{Non-shears Type}
\label{sec:nonshears}
Suppose $\pi(\mathbb{T})$ contains a translation hyperoval of non-shears type. By Lemma \ref{lem:dickson}, we require the existence of some $g\in \cL\backslash C(\SSS)$ such that $\rank(D_{R_y}^{-1} - D_g)\geq n-1$ for all $y\in \Fqn^\times$ and $\rank(D_g)=n-1$.

Similar to the shears case, we may assume without loss of generality that $g_0\in \{0,1,j,j^2\}$ and $g_1\in \FF_8$. 

Note if $R_y(x)=yx+jy^{2^4}x^{2^2}$ for $y\neq0$, then $R_y$ is $\mathbb{F}_4$-linear and so $R_{y}^{-1}$ must also be $\mathbb{F}_{4}$-linear. It is straighforward then to calculate $R_{y}^{-1}$, which we find to be
$$
    R_{y}^{-1}(x)=y^{62}j^{21}x+y^{11}j^{22}x^4+y^{59}j^{26}x^{16}.
$$

Due to the complexity of the coefficients of $\dinv_g(y)$ and the unknown element $a$ such that $\dinv_g(a)\ne 0$, there is little that can be done from a theoretical point of view utilising Lemma \ref{lem:dicksondet}, beyond the above restrictions on the coefficients of $g$. Hence we must rely on a long computation using Lemma \ref{lem:dickson}.




\begin{theorem}\label{thm:nonshearsnotexist}
    The Generalised Twisted Field plane of order $64$ does not contain a translation hyperoval of non-shears type.
\end{theorem}

\begin{proof}

The following MAGMA code verifies that there are no tuples $(g_0,g_1,g_2,g_3,g_4,g_5)$ with $g_0\in \{0,1,j,j^2\}, g_1\in \FF_8$, and $g_i\in \FF_{64}$ for $i=3,4,5,6$ such that $\rank(D_{R_y}^{-1}-D_g)\geq n-1$ for all $y\in \Fqn$ and $\rank(D_g)=n-1$.

{\small
\begin{quote}
\begin{verbatim}
q := 2;
n := 6;

F := GF(q);
P<x> := PolynomialRing(F);
L<j> := ext<F|x^6+x+1>;
F8 := {x:x in L|x^8 eq x};

DicksonMatrix := function(v,n,q);
return Matrix([Rotate([a^(q^i):a in v],i):i in [0..n-1]]);
end function;

Cinv := {DicksonMatrix([j^21*y^62,0,j^22*y^11,0,j^26*y^59,0],n,q):y in L};

time nonshears := {<g0,g1,g2,g3,g4,g5>:g0 in {0,1,j,j^2},g1 in F8,g2,g3,g4,g5 in L|
forall{z:z in Cinv|Rank(z-f) ge n-1} where f is DicksonMatrix([g0,g1,g2,g3,g4,g5],n,q)};
#nonshears eq 0;

\end{verbatim}
\end{quote}}

Hence by Lemma \ref{lem:dickson}, there does not exist a translation hyperoval of non-shears type in this plane.
\end{proof}

The calculation used in this proof takes approximately 8.5 hours on a single CPU. We note that this computation could clearly be parallelised and optimised further, but we do not attempt any improvements beyond the above restrictions on $g_0$ and $g_1$. Without these restrictions, the computation would take approximately three weeks.

\section{Conclusion and Remarks}

This culminates in the following theorem, disproving Cherowitzo's conjecture.

\begin{theorem}
    There does not exist a translation hyperoval in the Twisted Field Plane of order $64$.
\end{theorem}

\begin{remark}
Due to the previously described equivalences, we have also demonstrated the existence of a $6$-spread in $V(12,2)$ not admitting a scattered subspace of dimension $6$, and an MRD code (semifield spread set) in $M_6(\FF_2)$ with minimum distance $6$ and covering radius less than $5$.
\end{remark}



\begin{remark}
    The situation for the remaining semifield planes of order $64$ is more difficult to analyse theoretically. Instead we would need to rely on exhaustive computer searches. For translation hyperovals of shears type, this can be done relatively efficiently by exploiting the additivity of $C(\SSS)$; a naive implementation can perform an exhaustive search in about 8 hours (as opposed to less than a second for the generalised twisted field). In fact, it turns out that many semifield planes of order $64$ do not contain a translation hyperoval of shears type.

    However, for the non-shears case we do not have additivity, and for the majority of semifields we do not have enough symmetries to constrain the coefficients $g_i$ as in Section \ref{sec:nonshears}, and so exhaustive computation takes much longer. Hence further theoretical reductions, or a more significant parallelised computation, would be necessary in order to determine the existence or non-existence of translation hyperovals for these planes.
\end{remark} 





\begin{remark}
    Although hyperovals cannot exist in planes of odd order, scattered subspaces of maximum dimension with respect to spreads can still exist. The corresponding point set in the associated projective plane is a set of $q$ points not contained in the translation line meeting each line in $0,1$ or $q$ points upon which a group of translations acts transitively.

We can repeat the arguments from this paper in part; however, since $\frac{q^n-1}{q-1}<q^n-1$ for $q>2$, we cannot conclude much about $d_f(y)$. It remains an open question whether or not spreads defined by generalised twisted fields possess a scattered subspace of dimension $n$ for general $q$.
\end{remark}

\bibliographystyle{abbrv}
\bibliography{translationrefs}



\end{document}